\documentclass[11pt,en]{elegantpaper}

\title{Nonlinear conjugate gradient method for vector optimization on Riemannian manifolds with retraction and vector transport}

\author{Kangming Chen$^\ast$ \and Ellen H. Fukuda$^\ast$ \and Hiroyuki Sato%
  \thanks{Graduate School of Informatics, Kyoto University, Kyoto \mbox{606--8501}, Japan
  (\texttt{kangming@amp.i.kyoto-u.ac.jp, ellen@i.kyoto-u.ac.jp, hsato@i.kyoto-u.ac.jp}).}}

\date{}

\usepackage{array}
\usepackage{amsmath}
\usepackage{amsthm}
\usepackage{algorithmic}
\usepackage{graphicx,xcolor}
\usepackage{subfigure}
\usepackage[ruled, boxed, noend, norelsize]{algorithm2e}
\usepackage{multirow}

\newtheorem{assumption}{\rm\textbf{Assumption}}

\newcommand{\FR}{\mathrm{FR}}
\newcommand{\CD}{\mathrm{CD}}
\newcommand{\DY}{\mathrm{DY}}
\newcommand{\PRP}{\mathrm{PRP}}
\newcommand{\HS}{\mathrm{HS}}
\newcommand{\LS}{\mathrm{LS}}

\renewcommand{\Re}{\mathbb{R}} 
\newcommand{\norm}[1]{\left\|#1\right\|}
\DeclareMathOperator{\interior}{int}
\DeclareMathOperator{\conv}{conv}
\DeclareMathOperator{\cone}{cone}
\DeclareMathOperator{\Image}{Image}

\usepackage[backend=biber]{biblatex}
\begin{document}

\maketitle

\begin{abstract}
In this paper, we propose nonlinear conjugate gradient methods for vector optimization on Riemannian manifolds. The concepts of Wolfe and Zoutendjik conditions are extended for Riemannian manifolds. Specifically, we establish the existence of intervals of step sizes that satisfy the Wolfe conditions. The convergence analysis covers the vector extensions of the Fletcher--Reeves, conjugate descent, and Dai--Yuan parameters. Under some assumptions, we prove that the sequence obtained by the algorithm can converge to a Pareto stationary point. Moreover, we also discuss several other choices of the parameter. Numerical experiments illustrating the practical behavior of the methods are presented.

\end{abstract}





\section{Introduction}

Manifold optimization refers to optimization problems where the search space is a manifold instead of Euclidean space. This approach has gained significant attention in recent years due to its several advantages over traditional optimization methods. 
In particular, a constrained Euclidean optimization problem can become an unconstrained one in manifold optimization if the set of constraints forms a Riemannian manifold.
There has been a growing body of research on manifold optimization. In particular, numerous algorithms and techniques have been developed, including gradient-based methods and trust-region methods, tailored specifically for optimization on manifolds~\cite{AbsMahSep2008, boumal2023intromanifolds, absil_trust-region_2007}. 

On the other hand, vector optimization refers to the minimization or maximization of vector-valued functions with respect to a cone, where the aim is to find a set of solutions that simultaneously optimize multiple objective functions. The well-known multiobjective optimization is a particular case of vector optimization, where the cone is the nonnegative orthant. Unlike scalar optimization, vector and multiobjective optimization consider the trade-offs and compromises among different objectives. It aims to identify efficient solutions that lie on the Pareto front, representing the best possible trade-off between conflicting objectives. Besides the traditional scalarization methods and metaheuristics, several descent methods have been proposed for vector optimization~\cite{fukuda_inexact_2013, grana_drummond_steepest_2005, tanabe_proximal_2019, chen_conditional_2023}.

In this paper, we deal with the combination of vector optimization and Riemannian manifold optimization. In the literature, we can observe that the multiobjective or vector steepest descent method \cite{bento_unconstrained_2012, bento_inexact_2013}, the subgradient method \cite{bento_subgradient_2013}, and the proximal point method \cite{bento_proximal_2018} have been extended to Riemannian manifolds. Now increasing attention has been given to this field, and here we study the so-called nonlinear conjugate gradient methods. We remark that the conjugate gradient method was originally introduced by Fletcher and Reeves in \cite{flecher1964function}, while the Riemannian conjugate gradient method has been discussed in several studies, including \cite{ sato_riemannian_2021_article, sakai2021sufficient, zhu2017riemannian}. On the other hand, the nonlinear conjugate gradient method for vector optimization was first proposed in~\cite{lucambio_perez_nonlinear_2018}, and some other choices of parameters were considered later in~\cite{goncalves_2020,goncalves_study_nodate}. 

Here, we propose a nonlinear conjugate gradient method for vector optimization on Riemannian manifolds. We note, however, that an unpublished work by Ferreira, Lucambio-P\'erez, and Prudente~\cite{ferreira_unpublished} also considered those methods, with a specific choice of retractions and vector transports. In particular, they used the exponential map and the parallel translations, which may not be easy to compute numerically. During the writing of our paper,
a similar work, that uses general retractions and vector transports but for multiobjective problems, was also done in \cite{najafi_multiobjective_2023}. They establish the convergence of the algorithm for the Fletcher--Reeves (FR) and Dai--Yuan (DY) parameters with the corresponding sufficient descent condition and a desired inequality. In this paper, we do not rely on satisfying these conditions. Instead, we demonstrate that our algorithm can obtain descent directions and establish its convergence not only under FR and DY but also conjugate descent (CD) parameters. Moreover, we also discuss several other choices of the parameter.

This paper is organized as follows. In Section~\ref{sec:preliminaries}, we review some basic knowledge and results used in vector optimization and Riemannian optimization. In Section~\ref{sec:NVRCG}, we propose an algorithm of nonlinear conjugate gradient method for vector optimization on Riemannian manifolds (NVRCG), using retractions and vector transports. In Section~\ref{sec:convergence}, we extend Zoutendjik’s type condition to Riemannian manifolds and analyze the convergence of the NVRCG method with the Riemannian extensions of FR, CD, and DY parameters. We also give simple results for Polak--Ribière--Polyak (PRP), Hestenes--Stiefel (HS), and Liu--Storey (LS) parameters. Some simple numerical experiments are presented in Section~\ref{sec:experiments}. Finally, in Section~\ref{sec:conclusion} we conclude this paper and give potential future directions.

\section{Preliminaries}
\label{sec:preliminaries}

In this section, we summarize some definitions, results, and important notions used in vector optimization and Riemannian optimization. We refer to \cite{grana_drummond_steepest_2005, drummond_projected_2004, VectorOptimization} for details related to vector optimization. 
Here, let $K\subset \mathbb{R}^m$ be a closed, convex, and pointed (i.e., $K \cap (-K) = \{ 0 \}$) cone with a nonempty interior.
The partial order in $\mathbb{R}^m$ induced by $K$ is denoted by $\preceq_K$, and defined by
$$
u \preceq_K v \Leftrightarrow v-u \in K,
$$
and the relation $\prec_K$ induced by $\interior(K)$, the interior of $K$, is defined by
$$
u \prec_K v \Leftrightarrow v-u \in \operatorname{int}(K) .
$$ 
Let $C \subset K^* \setminus \{0\}$ be compact and such that 
$$
K^* = \cone(\conv(C)),
$$
that is, the cone generated by the convex hull of $C$ is equal to $K^*$, where $K^* :=\{ \omega \in\mathbb{R}^m\mid \omega^T y \geq 0 \mbox{ for all } y \in K \}$ is the positive polar (or dual) cone of $K$. 
 In general, we note that a smaller set $C$
can be chosen. For instance, in classical optimization, we have $K =\Re_+ $, then $ K^* = \Re_+$ and $C = \{1\}$. In multiobjective optimization, $K = K^* = \Re^m_+$, and thus $C$ can be taken as the canonical basis.
In the general case, we can specifically take 
$$ 
C= \{ \omega \in K^* \mid \norm{\omega} =1\},
$$
where $\norm{\cdot}$ denotes the Euclidean norm.

Now, consider the following vector optimization problem:
$$
\underset{K}{\operatorname{min}} \, F(x) \quad \mbox{s.t. } x \in \mathbb{R}^n,
$$
where $F \colon \mathbb{R}^n \rightarrow \mathbb{R}^m$ is a continuously differentiable function, and the minimization is with respect to the cone $K$. As in multiobjective optimization, there is not always a point that minimizes all the objectives at once, so we use the concept of Pareto optimality. A point $x^* \in \mathbb{R}^n$  is called $K$-Pareto optimal  if there exists no other $x \in \mathbb{R}^n$ with $F(x) \preceq_K F\left(x^*\right)$ and $F(x) \neq F\left(x^*\right)$, or weakly $K$-Pareto  optimal  if there exists no other $x \in \mathbb{R}^n$ with $F(x) \prec_K F\left(x^*\right)$.  It can be readily observed that $K$-Pareto optimal points are always weakly $K$-Pareto optimal. However, it should be noted that the converse does not hold universally.
Moreover, a necessary condition for the $K$-optimality of $x^*$ is
$$
-\operatorname{int}(K) \cap \operatorname{Image}\left(J F\left(x^*\right)\right)=\emptyset,
$$
where $J F(x)$ denotes the Jacobian of $F$ at $x$, and $\Image$ means the image on $\mathbb{R}^m$. A point $x^* \in \mathbb{R}^n$ that satisfies the above condition is called $K$-critical (or $K$-stationary) for $F$. Therefore, if $x$ is not $K$-critical, there exists $d \in \mathbb{R}^n$ such that $J F(x) d \in-\operatorname{int}(K)$. Every such vector $d$ is a $K$-descent direction for $F$ at $x$, i.e., there exists $T>0$ such that $0<t<T$ implies that $F(x+t d) \prec_K F(x)$. 
Here, the issue of finding optimal points (weakly or otherwise) for the problem induced by $K = \Re^n$ is known as a multiobjective optimization one. In this case,  $x \in \mathbb{R}^n$ is a critical point if and only if $J F(x) d \nprec 0$ for all $d \in \mathbb{R}^n$, namely,
$$
\max _{i \in \{1, \ldots, m\} } \nabla F_i(x)^T d \geq 0, \quad \mbox{ for all } d \in \mathbb{R}^n.
$$

Now we give some important knowledge related to Riemannian manifolds \cite{AbsMahSep2008, boumal2023intromanifolds, sato_riemannian_2021}.
A Riemannian manifold $\mathcal{M}$ is a manifold endowed with a Riemannian metric 
$(\eta_x, \sigma_x) \mapsto \langle\eta_x, \sigma_x\rangle_x \in \mathbb{R} $, where $\eta_x$ and $\sigma_x$ are tangent vectors in the tangent space of $\mathcal{M}$  at $x$.
The tangent space of a manifold $\mathcal{M}$ at $x \in \mathcal{M}$ is denoted as $T_x \mathcal{M}$, 
and the tangent bundle of $\mathcal{M}$ is denoted as $T \mathcal{M}:=\left\{(x, d) \mid d \in T_x \mathcal{M}, x \in \mathcal{M}\right\}$. The norm of $\eta \in T_x \mathcal{M}$ is defined as $\|\eta\|_x:=\sqrt{\langle\eta, \eta\rangle_x}$. 
For a map $F: \mathcal{M} \rightarrow \mathcal{N}$ between two manifolds $\mathcal{M}$ and $\mathcal{N}, \mathrm{D} F(x): T_x \mathcal{M} \rightarrow T_{F(x)} \mathcal{N}$ denotes the derivative of $F$ at $x \in \mathcal{M}$. 
The Riemannian gradient grad $f(x)$ of a function $f: \mathcal{M} \rightarrow \mathbb{R}$ at $x \in \mathcal{M}$ is defined as a unique tangent vector at $x$ satisfying $\langle\operatorname{grad} f(x), \eta\rangle_x=\mathrm{D} f(x)[\eta]$ for any $\eta \in T_x \mathcal{M}$.

In Riemannian optimization, we use a retraction to project points from the tangent space of the manifold onto the manifold itself. 

\begin{definition}\cite{sato_riemannian_2021}
    A smooth map  $R \colon T \mathcal{M} \rightarrow \mathcal{M}$ is called a retraction on a smooth manifold $\mathcal{M}$ if  the restriction of $R$  to the tangent space $T_x \mathcal{M}$ at any point $x \in \mathcal{M}$, denoted by $R_x$, satisfies  the following conditions:
    \begin{enumerate}
    \item $R_x\left(0_x\right)=x$,
    \item $\mathrm{D} R_x\left(0_x\right)=\operatorname{id}_{T_x \mathcal{M}}$ for all $x \in \mathcal{M}$, 
    \end{enumerate}
    where $0_x$ and $\operatorname{id}_{T_x \mathcal{M}}$ are the zero vector of $T_x \mathcal{M}$ and identity map in $T_x \mathcal{M}$, respectively.
\end{definition}
The concept of retraction plays a fundamental role in optimization algorithms by providing a way to move from the tangent space back to the manifold, ensuring that iterates remain feasible solutions. Different types of retractions can be used depending on the problem and the specific properties of the manifold. Many Riemannian optimization algorithms use a retraction that is a generalization of the exponential map 
on $\mathcal{M}$.
For instance, assume that we have an iterative method generating iterates $\{x^k\}$. In Euclidean spaces, the update takes the form $x_{k+1}=x_k + t_k d _k$, while in the Riemannian case it is generalized as
$$x_{k+1}=R_{x_k}(t_k d _k),\quad \text{for }  k = 0,1,2, \ldots,$$
where $d_k$ is a descent direction and $t_k$ is a step size. 

Another important concept is the vector transport, which plays a crucial role in Riemannian optimization by ensuring that computations are performed in a manner that preserves the manifold's intrinsic geometry. By leveraging transport, optimization algorithms can effectively navigate the manifold's curved spaces, compute relevant geometric quantities, and converge toward optimal solutions.
\begin{definition}\cite{sato_riemannian_2021}
A map $\mathcal{T} \colon T \mathcal{M} \oplus T \mathcal{M} \rightarrow T \mathcal{M}$ is called a vector transport on $\mathcal{M}$ if it satisfies the following conditions, where $T \mathcal{M} \oplus T \mathcal{M}:=\left\{(\xi, d) \mid \xi, d \in T_x \mathcal{M}, x \in \mathcal{M}\right\}$ is the Whitney sum:
\begin{enumerate}
\item There exists a retraction $R$ on $\mathcal{M}$ such that $\mathcal{T}_d(\xi) \in T_{R_x(d)} \mathcal{M}$ for all $x \in \mathcal{M}$ and $\xi, d \in T_x \mathcal{M}$.
\item For any $x \in \mathcal{M}$ and $\xi \in T_x \mathcal{M}, \mathcal{T}_{0_x}(\xi)=\xi$ holds, where $0_x$ is the zero vector in $T_x \mathcal{M}$, i.e., $\mathcal{T}_{0_x}$ is the identity map.
\item For any $a, b \in \mathbb{R}, x \in \mathcal{M}$, and $\xi, d, \zeta \in T_x \mathcal{M}, \mathcal{T}_d(a \xi+b \zeta)=a \mathcal{T}_d(\xi)+b \mathcal{T}_d(\zeta)$ holds, i.e., $\mathcal{T}_d$ is a linear map from $T_x \mathcal{M}$ to $T_{R_x(d)} \mathcal{M}$.
\end{enumerate}
\end{definition}

Note that a map $\mathcal{T}$ defined by $\mathcal{T}_d(\xi):=\mathrm{P}_{\gamma_{x, d}}^{1 \leftarrow 0}(\xi)$ is also a vector transport, 
where $\mathrm{P}_{\gamma_{x, d}}^{1 \leftarrow 0}$ is the parallel translation along the geodesic $\gamma_{x, d}(t):=\operatorname{Exp}_x(t d)$ connecting $\gamma_{x, d}(0)=x$ and $\gamma_{x, d}(1)=\operatorname{Exp}_x(d)$ with the exponential map Exp as a retraction.

\section{Nonlinear vector Riemannian conjugate gradient method}
\label{sec:NVRCG}

We consider the following unconstrained vector optimization problem:
\begin{equation}\label{prob:1}
    \begin{array}{ll}
    \underset{K}{\operatorname{min}} & F(x) \\
    \text { s.t. } & x \in \mathcal{M},
    \end{array}
\end{equation}
where $F:\mathcal{M} \to \Re^m$ is continuously differentiable and $\mathcal{M}$ is an  $n$-dimensional smooth manifold with a Riemannian metric $\left \langle \cdot, \cdot  \right \rangle$. 
Define the function $\varphi: \mathbb{R}^m \rightarrow \mathbb{R}$ by
$$
\varphi(y) :=\sup \{ y^Tw \mid w \in C\}.
$$
In view of the compactness of $C$, where $C$ is given in Section \ref{sec:preliminaries}, $ \varphi$ is well defined. It can easily be seen that when $ m= 1$, $ \varphi(y)= y$ holds.
Then, we have
$$
\varphi(\mathrm{D} F(x) d)=\sup \{[\mathrm{D} F(x) d]^T w \mid w \in C\}.
$$
Note that it can be viewed as a nonlinear programming problem when $K=\mathbb{R}_{+}^m$.
The function $\varphi$ has some useful properties, that we recall below.
\begin{lemma}\cite[Lemma 3.1]{grana_drummond_steepest_2005}\label{lemmaLip}
Let $y, y^{\prime} \in \mathbb{R}^m$. Then, the following statements hold:
\begin{itemize}
    \item[(1)] $\varphi\left(y+y^{\prime}\right) \leq \varphi(y)+\varphi\left(y^{\prime}\right)$ and $\varphi(y)-\varphi\left(y^{\prime}\right) \leq \varphi\left(y-y^{\prime}\right)$;

    \item[(2)] If $y \preceq_K y^{\prime}$, then $\varphi(y) \leq \varphi\left(y^{\prime}\right)$; if $y \prec_K y^{\prime}$, then $\varphi(y)<\varphi\left(y^{\prime}\right)$;

    \item[(3)] $\varphi$ is Lipschitz continuous with constant $1$ .
\end{itemize}
\end{lemma}

\noindent From the definition of function  $\varphi$, we note that:
\begin{itemize}
        \item $d$ is a $K$-descent direction for $K$ at $x$ if and only if $\varphi(\mathrm{D} F(x) d)< 0$,
        \item $x$ is $K$-critical if and only if $\varphi(\mathrm{D} F(x) d) \geq 0$ for all $d$.
\end{itemize}

Now, consider the extension of the notion of the steepest descent direction to the vector-valued case. We denote it as $v \colon \mathbb{R}^n \rightarrow \mathbb{R}^n$ and define it as
\begin{equation}
    \label{eq:v(x)}
    v(x) := {\arg \min}_d \left\{\varphi(\mathrm{D} F(x) d)+\frac{\|d\|^2}{2} \, \middle| \, d \in \mathbb{R}^n \right\}.
\end{equation}
Also, let $\theta \colon \mathbb{R}^n \rightarrow \mathbb{R}$ be defined as $\theta(x) := \varphi(\mathrm{D}F(x)v(x))+\|v(x)\|^2 / 2$. We note that when $m = 1$ and $C = \{1\}$, we have $v(x)= -\nabla f(x)$, that is, we retrieve the classical steepest descent direction in nonlinear optimization.

\begin{remark}
For multiobjective optimization, where $K=\mathbb{R}_{+}^m$, with $C$ given by the canonical basis of $\mathbb{R}^m$,  the direction $v(x)$ can be computed by solving
$$
\begin{array}{ll}
\operatorname{minimize} & \alpha+\frac{1}{2}\|d\|^2 \\
\text { subject to } & {[J F(x) d]_i \leq \alpha, \quad i=1, \ldots, m,}
\end{array}
$$
which is a convex quadratic problem with linear inequality constraints.
\end{remark}

With the above definition, we can easily obtain the following lemma.
\begin{lemma}\cite[lemma 3.3]{grana_drummond_steepest_2005}
    \begin{itemize}
        \item[(1)] If $x$ is $K$-critical, then $v(x)=0$ and $\theta(x)=0$.
        \item[(2)] If $x$ is not $K$-critical, then $v(x) \neq 0$, $\theta(x)<0$, $\varphi(\mathrm{D} F(x) v(x))<-\frac{\|v(x)\|^2}{2}<0$, and $v(x)$ is a $K$-descent direction for $F$ at $x$.
        \item[(3)] The mappings $v$ and $\theta$ are continuous. 
    \end{itemize}
\end{lemma}


In order to get a proper decrease, line search procedures are crucial. In particular, for nonlinear conjugate gradient methods, the so-called Wolfe conditions are widely used, and thus we extend them to vector Riemannian manifolds.
Let $e \in K$ be given such that 
$$
0 < \langle\omega, e \rangle\leq 1\quad  \text{for all} \quad \omega \in C.
$$
In particular, for vector optimization, where $ K = \mathbb{R}^m_+$ and $C$ is given by the canonical basis of $\mathbb{R}^m$, we can take $e= [1,\ldots , 1]^T\in \mathbb{R}^m$. Letting $0< c_1< c_2< 1 $, we propose the (weak) Wolfe conditions as follows:
\begin{align}
    & F\left(R_{x_k}\left(t_k d_k\right)\right) \preceq_K F\left(x_k\right)+c_1 t_k \varphi\left(\mathrm{D}F(x_k)[d_k]\right) e, \label{eq:armijo} \\
    & \varphi\left(\mathrm{D}F(R_{x_k}\left(t_kd_k\right))[\mathrm{D} R_{x_k}\left(t_k d_k\right)\left[d_k\right]]\right) \geq c_2 \varphi\left(\mathrm{D}F(x_k)[d_k]\right),
    \label{eq:wolfe}
\end{align} 
and the strong Wolfe conditions are given by
\begin{align}
    & F\left(R_{x_k}\left(t_k d_k\right)\right) \preceq_K F\left(x_k\right)+c_1 t_k \varphi\left(\mathrm{D}F(x_k)[d_k]\right) e, \nonumber \\
    & | \varphi\left(\mathrm{D}F(R_{x_k}\left(t_kd_k\right))[\mathrm{D} R_{x_k}\left(t_k d_k\right)\left[d_k\right]]\right) | \leq c_2 |\varphi\left(\mathrm{D}F(x_k)[d_k]\right)|.
    \label{eq:strong_wolfe}
\end{align}
Note that the first condition is the same for both (weak) Wolfe and strong Wolfe conditions, and it is usually called Armijo condition.


In the Euclidean case, the search direction of the nonlinear vector conjugate gradient method is given by 
$$
d_{k+1}=v\left(x_{k+1}\right)+\beta_{k+1}d_{k},
$$
for $k \geq 0$, where $\beta_{k+1} \in \Re$ is a scalar algorithmic parameter.
To extend it to the Riemannian case, in this paper, we use a vector transport called the differentiated retraction \cite{sato_riemannian_2021}.
Using $\mathcal{T}^k: T_{x_k} \mathcal{M} \rightarrow T_{x_{k+1}} \mathcal{M}$ with $\mathcal{T}^k(d_{k}):= \mathcal{T}_{t_kd_k}^R(d_k)= \mathrm{D} R_{x_k}\left(t_k d_k\right)\left[d_k\right]$, we have
$$
d_{k+1}=v\left(x_{k+1}\right)+\beta_{k+1} \mathcal{T}^k(d_{k}).
$$
For simplicity, we also define 
$$
\psi_{x,d}(td):= \varphi(\mathrm{D} F(R_x(td))[\mathrm{D}R_x(td)[d]]),
$$ 
and then $\psi_{x_k,d_k}(0) = \varphi(\mathrm{D}F(x_k)[d_k]).$

In the case of single-objective optimization, serval types of $\beta$ are proposed \cite{flecher1964function, fletcher1980unconstrained, dai2000convergence, polak1969note, stiefel1952methods, liu1991efficient}. Here, we extend them to vector optimization on Riemannian manifolds as follows:
$$ \beta_k^{\FR} = \frac{\psi_{x_k,v_k}(0)}{\psi_{x_{k-1},v_{k-1}}(0)} ,$$
$$ \beta_k^{\CD} = \frac{\psi_{x_k,v_k}(0)}{\psi_{x_{k-1},d_{k-1}}(0)} ,$$
$$ \beta^{\DY}_k=\frac{-\psi_{x_k,v_k}(0)}{\psi_{x_{k-1},d_{k-1}}(t_{k-1}d_{k-1})-\psi_{x_{k-1},d_{k-1}}(0)},$$
$$
\beta^{\PRP}_k = \frac{-\psi_{x_k,v_k}(0)+\psi_{x_{k},\mathcal{T}^{k-1}{(v_k)}}(0)}{-\psi_{x_{k-1},v_{k-1}}(0)},
$$
$$ \beta_k^{\LS} =\frac{-\psi_{x_k,v_k}(0)+\psi_{x_{k-1},\mathcal{T}^{k-1}{(v_k)}}(0)}{-\psi_{x_{k-1},d_{k-1}}(0)} ,$$
$$
\beta^{\HS}_k = \frac{-\psi_{x_k,v_k}(0)+\psi_{x_{k-1},\mathcal{T}^{k-1}{(v_k)}}(0)}{\psi_{x_{k},\mathcal{T}^{k-1}{(d_{k-1})}}(0)-\psi_{x_{k-1},d_{k-1}}(0)}.
$$


Now we state our nonlinear conjugate gradient method for vector optimization problem \eqref{prob:1} on Riemannian manifolds in Algorithm~\ref{alg:nvrcg}. In each iteration, the algorithm first computes the steepest descent direction by solving the subproblem \eqref{eq:v(x)} and then calculates the actual descent direction according to the updated formula, which also depends on the choice of $\beta$. Steps 3 and 4 give the proper step size and update the iteration. We will give a convergence analysis of the algorithm in the next section.

\begin{algorithm}[h]
  \caption{Nonlinear vector Riemannian conjugate gradient method (NVRCG)}
  \label{alg:nvrcg}
  \LinesNotNumbered
    \noindent Step 0. Let $x_0 \in \mathcal{M}$ and initialize $k \leftarrow 0$.\\
    Step 1. Compute $v\left(x_{k}\right)$ as in~\eqref{eq:v(x)}. If $v\left(x_k\right)=0$, then stop.\\
    Step 2. Compute
    $$
    d_k= \begin{cases}v\left(x_k\right), & \text { if } k=0, \\ v\left(x_k\right)+\beta_k \mathcal{T}^{k-1}(d_{k-1}), & \text { if } k \geq 1,\end{cases}
    $$
    where $\beta_k$ is an algorithmic parameter.\\
    Step 3. Compute a step size $t_k>0$ by a line search procedure and set $x_{k+1}=R_{x_k}(t_kd_k)$. \\
    Step 4. Set $k \leftarrow k+1$, and go to Step 1.
\end{algorithm}

\section{Global convergence analysis}
\label{sec:convergence}

In this section, we prove the convergence of the algorithm with each type of $\beta$.
First, we propose some assumptions for convergence analysis.
\begin{assumption}
\label{assm1}
There exist constant $L>0$, for all $x \in \mathcal{M}, d \in T_x \mathcal{M}$ with $\|d\|=1, t \geq 0$, 
such that
$$\left\|\mathrm{D}\left(F \circ R_x\right)(td)[d]-\mathrm{D}\left(F \circ R_x\right)(0)[d]\right\| \leq L t. $$
\end{assumption}
Remark: When $F=[f_1,\ldots, f_m]^T$, we can simply give such assumption for each $f_i$, and take $L = \max \{L_i, i =1, \ldots, m \}$.

\begin{assumption}
\label{assum2}
    Set $\mathcal{L} = \{ x\mid F(x)\preceq_K F(x_0)\}$.
    All monotonically nonincreasing sequences in $F(\mathcal{L})$ are bounded from below, i.e., 
    if $\left\{G_k\right\}_{k \in \mathbb{N}} \subset F(\mathcal{L})$ and $G_{k+1} \preceq_K G_k$ for all $k$, then there exists $\mathcal{G} \in \mathbb{R}^m$ such that $\mathcal{G} \preceq_K G_k$ for all $k$.
\end{assumption}

With Assumptions \ref{assm1} and \ref{assum2}, we can prove that the method we propose satisfies Zoutendjik’s type condition when $d_k$ is a $K$-descent direction. We will prove $d_k$ is a $K$ descent direction in each specific type of $\beta$ later.
\begin{proposition}\label{Zoutendijk}
Suppose that Assumptions \ref{assm1} and \ref{assum2} hold. Consider the iteration, where $d_k$ is a $K$-descent direction for $F$ at $x_k$, and $t_k$ satisfies the Wolfe conditions. Then, we have
$$
\sum_{k \geq 0} \frac{\varphi^2\left(\mathrm{D}F(x_k)[d_k]\right)}{\left\|d_k\right\|^2}<\infty.
$$
\end{proposition}
\begin{proof}
Under Assumption \ref{assm1}, from the Wolfe conditions and Lemma \ref{lemmaLip}, we have
\begin{equation}\label{extentLip}
    \begin{aligned}
    &\left(c_2-1\right)\varphi\left(\mathrm{D} F\left(x_k\right) [d_k]\right)\\
    &\leq \varphi\left(\mathrm{D} F\left(R_{x_k}\left(t_k d_k\right) [\mathrm{D} R_{x_k}\left(t_k d_k\right)\left[d_k\right]]\right)-\varphi\left(\mathrm{D} F\left(x_k\right) [d_k]\right)\right. \\
    &\leq\left\|\mathrm{D} F\left(R_{x_k}\left(t_k d_k\right)\right) \mathrm{D} R_{x_k}\left(t_k d_k\right)\left[d_k\right]-\mathrm{D} F\left(x_k\right) d_k\right\| \\
    &=\left\|\mathrm{D}\left(F \circ R_{x_k}\right)\left(t_k d_k\right)\left[d_k\right]-\mathrm{D}\left(F \circ R_{x_k}\right)(0)\left[d_k\right]\right\| \\
    &\leq\|d_k\| \bigg\| \mathrm{D}\left(F \circ R_{x_k}\right)\left(t_k\left\|d_k\right\| \frac{d_k}{\left\|d_k\right\|}\right)\left[\frac{d_k}{\left\|d_k\right\|}\right] -\mathrm{D}\left(F \circ R_{x_k}\right)(0)\left[\frac{d_k}{\left\|d_k\right\|}\right] \bigg\| \\
    &\leq L t_k\left\|d_k\right\|^2.
    \end{aligned}
\end{equation}

The first inequality is derived from the Wolfe condition \eqref{eq:wolfe}, while the second inequality is obtained through the utilization of Lemma \ref{lemmaLip}(3). The equality arises from the utilization of the chain rule for complex functions. Additionally, the third inequality is derived based on the fundamental definition of the norm, and the final inequality is a direct consequence of Assumption \ref{assm1}.

Since $\varphi\left(\mathrm{D}F(x_k)[d_k]\right)<0$ and $\left\|d_k\right\| \neq 0$, we obtain
$$
\frac{\varphi^2\left(\mathrm{D}F(x_k)[d_k]\right)}{\left\|d_k\right\|^2} \leq L t_k \frac{\varphi\left(\mathrm{D}F(x_k)[d_k]\right)}{c_2-1}.
$$
By the Wolfe conditions, we have that $\left\{F\left(x_k\right)\right\}_{k \geq 0}$ is monotonically decreasing and that
$$
F\left(x_{k+1}\right)-F\left(x_0\right) \preceq_K c_1 \sum_{j=0}^k t_j \varphi\left(\mathrm{D}F(x_j)d_j\right) e
$$
for all $k \geq 0$. Then, by the assumption, since $\left\{F\left(x_k\right)\right\}_{k \geq 0} \subset \mathcal{L}$, there exists $\mathcal{F} \in \mathbb{R}^m$ such that
$$
\mathcal{F}-F\left(x_0\right) \preceq_K c_1 \sum_{j=0}^k t_j \varphi\left(\mathrm{D}F(x_j)d_j\right) e
$$
for all $k \geq 0$. Therefore, for all $k \geq 0$ and $w \in C$, we obtain
$$
\left\langle\mathcal{F}-F\left(x_0\right), w\right\rangle \leq c_1\langle e, w\rangle \sum_{j=0}^k t_j \varphi\left(\mathrm{D}F(x_j)d_j\right).
$$
This together with $c_2 \in  (0, 1) $ yields
$$
\frac{1}{c_2-1} \min \left\{\left\langle\mathcal{F}-F\left(x_0\right), \bar{w}\right\rangle \mid \bar{w} \in C\right\} \geq c_1\langle e, w\rangle \sum_{j=0}^k t_j \frac{\varphi\left(\mathrm{D}F(x_j)d_j\right)}{c_2-1}>0 .
$$
Therefore, we conclude that
$$
\sum_{k \geq 0} t_k \frac{\varphi\left(\mathrm{D}F(x_k)[d_k]\right)}{c_2-1}<\infty,
$$
which applies
$$
\sum_{k \geq 0} \frac{\varphi^2\left(\mathrm{D}F(x_k)[d_k]\right)}{\left\|d_k\right\|^2}<\infty.
$$
\end{proof}

In the following, we prove that there exist intervals of step sizes satisfying both the weak and the strong Wolfe conditions.
\begin{proposition} 
    Let F be of class $C^1$ defined on a smooth Riemannian manifold. Assume that $d$ is a $K$-descent direction for $F$ at $x$ and $F$ is lower bounded. Then, there exists $t>0$ that satisfies the Wolfe and strong Wolfe conditions.
\end{proposition}
\begin{proof}
    For given $w\in C$, define $\phi_w, l_w \colon \mathbb{R} \rightarrow \mathbb{R}$ by
    $$ \phi_w(t)=F(R_{x}(t d))^Tw $$
    and 
    $$ l_w(t)= F(x)^Tw + c_1 t \varphi\left(\mathrm{D}F(x)[d]\right).$$
    Observe that $\phi_w(0)=l_w(0)$. Because $d$ is a $K$-descent direction, we have $ \phi_w(0)\leq \varphi\left(\mathrm{D}F(x)[d]\right) < c_1\varphi\left(\mathrm{D}F(x)[d]\right)<0$. 
    So there exists $t>0$ such that $F(R_{x}(t d))^Tw = F(x)^Tw + c_1 t \varphi\left(\mathrm{D}F(x)[d]\right)$. Let $\hat{t}$ be the smallest number among such $t$. Because  $ e^Tw\leq 1$, we have for any $ t \in (0, \hat{t})$, 
    \begin{align*}
    F(R_{x}(t d))^Tw &\leq F(x)^Tw + c_1 t \varphi\left(\mathrm{D}F(x)[d]\right) \\
    &\leq F(x)^Tw + c_1 t \varphi\left(\mathrm{D}F(x)[d]\right)e^Tw.
    \end{align*}
    Therefore condition~\eqref{eq:armijo} holds.

    Since $\phi_w(0)=l_w(0)$ and $\phi_w(\hat{t})=l_w(\hat{t})$, there must exist $ \tilde{t}\in (0, \hat{t} )$ such that
    $$\mathrm{D}\phi_w(\tilde{t})= c_1 \varphi\left(\mathrm{D}F(x)[d]\right). $$
    Then, we have
    $$\varphi\left(\mathrm{D}F(R_{x}\left(\tilde{t}d\right))[\mathrm{D} R_{x}\left(\tilde{t} d\right)\left[d\right]]\right) \geq c_1 \varphi\left(\mathrm{D}F(x)[d]\right) .$$
    From \eqref{extentLip}, we know that $\varphi\left(\mathrm{D}F(R_{x}\left({t}d\right))[\mathrm{D} R_{x}\left({t} d\right)\left[d\right]]\right)$ can be seen as a continuous function of $t$, and by the intermediate value theorem, there exists $\overline{t}\in (0,\tilde{t})$, and because $0<c_1<c_2<1$, we obtain
    $$0 > \varphi\left(\mathrm{D}F(R_{x}\left(\overline{t}d\right))[\mathrm{D} R_{x}\left(\overline{t} d\right)\left[d\right]]\right) = c_1 \varphi\left(\mathrm{D}F(x_k)[d_k]\right) \geq c_2 \varphi\left(\mathrm{D}F(x)[d]\right). $$
    Therefore, conditions~\eqref{eq:wolfe} and~\eqref{eq:strong_wolfe} hold.
\end{proof}
\begin{remark}
    This proposition just showed that there exists $t$  such that $\varphi\left(\mathrm{D}F(R_{x}\left(td\right))[\mathrm{D} R_{x}\left(t d\right)\left[d\right]]\right) < 0$, but this does not mean it is true for all $t$. The exact relationship remains to be discussed.
\end{remark}

Now just set $ v_k= v(x_k)$, and we notice that in similar methods it usually imports a sufficient descent condition \cite{lucambio_perez_nonlinear_2018, najafi_multiobjective_2023}, that is,  for all $k$, there exists a positive $\alpha$ such that
\begin{equation}
    \varphi(\mathrm{D}F(x_k)[d_k]) \leq \alpha\varphi(\mathrm{D}F(x_k)[v_k]).\label{des}
\end{equation}
\begin{remark}
\eqref{des} is a normal assumption in many descent methods. This constraint is also used in the proof of general $\beta$  in vector optimization \cite[Thm 4.2(i)]{lucambio_perez_nonlinear_2018}, but if we assume this, it is equivalent to assuming that $d_k$ is a $K$-descent direction, which we want to avoid. So we do not use it in the following part. 
\end{remark}


From the NVRCG algorithm step 2, we have
\begin{align*}
\varphi(\mathrm{D}F(x_k)[d_k]) &\leq \varphi(\mathrm{D}F(x_k)v_k) + \beta_k \varphi(\mathrm{D}F(x_k)\mathcal{T}^{k-1}(d_{k-1})) \\
    &= \varphi(\mathrm{D}F(x_k)v_k) + \beta_k \varphi(\mathrm{D}F(R(x_{k-1}))\mathcal{T}^{k-1}(d_{k-1})),
\end{align*}
namely,
\begin{equation}\label{ineq1}
    \psi_{x_k,d_k}(0) \leq \psi_{x_k,v_k}(0) +\beta_k \psi_{x_{k-1},d_{k-1}}(t_{k-1}d_{k-1}) .
\end{equation}

From this inequality, we can easily get the following proposition.

\begin{proposition}
    Assume that $\beta_k$ is defined as follows. Then $d_k$ is a $K$-descent direction.
    $$
    \beta_k \in \begin{cases}{[0, \infty)} & \text { if } \psi_{x_{k-1},d_{k-1}}(t_{k-1}d_{k-1}) \leq 0, \\ {\left[0,-\psi_{x_k,v_k}(0) / \psi_{x_{k-1},d_{k-1}}(t_{k-1}d_{k-1})\right)} & \text { if } \psi_{x_{k-1},d_{k-1}}(t_{k-1}d_{k-1})>0. \end{cases}
    $$
\end{proposition}

For the purpose of convergence, we also need the following assumption.
\begin{assumption}
\label{assum3}
    The level set $\mathcal{L}=\{x | F(x)\preceq_K F(x_0)\}$ is bounded.
\end{assumption}

Under Assumption \ref{assum3}, we can claim that $\{\psi_{x_k,v_k}(0)\}$ is bounded.

\begin{theorem}{\cite[Theorem 5.1]{bento_unconstrained_2012}}
    Each accumulation point of the sequence $\{x_k\}$ is a Pareto stationary point.
\end{theorem}
\begin{proof}
    The $v(x)$ we defined is a steepest descent direction, and the proof of this theorem is similar to that of Theorem 5.1 in \cite{bento_unconstrained_2012}.
\end{proof}

\subsection{NVRCG method with FR parameter}
\label{prop1FR}

Now we consider the vector extension of the FR parameter on Riemannian manifolds:
$$ \beta_k^{\FR} = \frac{\psi_{x_k,v_k}(0)}{\psi_{x_{k-1},v_{k-1}}(0)} .$$

Under some assumptions, the global convergence can be proved if the parameter $\beta $ is bounded in magnitude by $\beta_k^{\FR}$.

\begin{proposition}
\label{propFR}
    Assuming $\beta_k^{\FR} \geq \beta_k > 0$ and $t_k$ satisfies the Strong Wolfe conditions with $0< c_1< c_2\leq 1/(1+m_k) $, there exists $m_k>0$ at each iteration. If $v(x_k)\neq 0$ for each $k \geq 0$,  $d_k$ is a $K$-descent direction and the following 
    inequality holds:
    \begin{align}
        1-\frac{c_2}{1-c_2m_k}  \leq \frac{\psi_{x_{k},d_{k}}(0)}{\psi_{x_{k},v_{k}}(0)} \leq \frac{1}{1-c_2m_k}
    \end{align}
\end{proposition}
\begin{proof}
    According to the definition of $\psi_{x_{k},v_{k}}(0)$, we have $\psi_{x_{k},v_{k}}(0) <0$ for all $k$, and $\beta_k^{\FR}>0$.
    
    Now we prove the statement by induction.
    
    When $k=0$, we have $ \psi_{x_{0},d_{0}}(0) =\psi_{x_{0},v_0}(0) <  0$, then the inequality holds.

    Assume the inequality still holds for arbitrary $k -1$, namely,
    $$ 1-\frac{c_2}{1-c_2m_{k-1}}  \leq \frac{\psi_{x_{k-1},d_{k-1}}(0)}{\psi_{x_{k-1},v_{k-1}}(0)} \leq \frac{1}{1-c_2m_{k-1}}.    $$
Recalling \eqref{ineq1} and considering the definition of $\psi $ and $\psi$ is bounded, there must exists integer $m>0$ satisfying the inequality:
    \begin{equation}\label{mk}
    \psi_{x_k,d_k}(0) \geq \psi_{x_k,v_k}(0) - m \beta_k |\psi_{x_{k-1},d_{k-1}}(t_{k-1}d_{k-1})|.
    \end{equation}
Let $m$ be the smallest one among such integers and $m_k = \max \{m, m_{k-1}\}.$ 
Dividing by $ \psi_{x_k,v_k}(0)$, which is less than 0, on both sides and combining with the definition of $\beta^{\FR}$ and strong Wolfe condition \eqref{eq:strong_wolfe}, we have
    \begin{align*}
        \frac{\psi_{x_k,d_k}(0)}{\psi_{x_k,v_k}(0)} &\leq  \frac{\psi_{x_k,v_k}(0) - m \beta_k |\psi_{x_{k-1},d_{k-1}}(t_{k-1}d_{k-1})|}{\psi_{x_k,v_k}(0)} \\
            &= 1 - m\frac{|\beta_k \psi_{x_{k-1},d_{k-1}}(t_{k-1}d_{k-1})|}{\psi_{x_k,v_k}(0)} \\
            &\leq 1 + mc_2\beta_k\frac{\psi_{x_{k-1},d_{k-1}}(0)}{\psi_{x_k,v_k}(0)} \\
            &\leq 1 + mc_2\frac{\psi_{x_{k-1},d_{k-1}}(0)}{\psi_{x_{k-1},v_{k-1}}(0)} \\
            &\leq 1 + mc_2\frac{1}{1-c_2m_{k-1}} \\
            &\leq \frac{1}{1-c_2m_k}.
    \end{align*}

Recalling \eqref{ineq1} again, dividing $ \psi_{x_k,v_k}(0)$ on both sides, combining with the definition of $\beta^{\FR}$ and strong Wolfe conditions, we have 
    \begin{align*}
        \frac{\psi_{x_k,d_k}(0)}{\psi_{x_k,v_k}(0)} &\geq  \frac{\psi_{x_k,v_k}(0) +\beta_k \psi_{x_{k-1},d_{k-1}}(t_{k-1}d_{k-1})}{\psi_{x_k,v_k}(0)} \\
            &= 1 + \frac{\beta_k \psi_{x_{k-1},d_{k-1}}(t_{k-1}d_{k-1})}{\psi_{x_k,v_k}(0)} \\
            &\geq 1 - c_2\beta_k\frac{\psi_{x_{k-1},d_{k-1}}(0)}{\psi_{x_k,v_k}(0)} \\
            &\geq 1 - c_2\frac{\psi_{x_{k-1},d_{k-1}}(0)}{\psi_{x_{k-1},v_{k-1}}(0)} \\
            &\geq 1- c_2\frac{1}{1-c_2m_{k-1}}\\
            &\geq 1-\frac{c_2}{1-c_2m_{k}}. 
    \end{align*}

    This completes the proof.
\end{proof}

\begin{theorem}\label{ThmFR}
    Consider the NVRCG algorithm. Suppose that Assumptions \ref{assm1}, \ref{assum2} and \ref{assum3}  hold and that
    \begin{equation}\label{assumpDk}
    \sum_{k \geq 0} \frac{1}{\left\|d_k\right\|^2}=\infty.
    \end{equation}
    If  $0 < \beta_k \leq \beta_k^{\FR}$, the assumptions of Proposition \ref{propFR} hold 
      and $t_k$ satisfies the strong Wolfe conditions, 
     then $\liminf _{k \rightarrow \infty}\left\|v\left(x_k\right)\right\|=0$.
\end{theorem}

\begin{proof}   
    From the definition of $v(x)$, we can easily observe $\sum_{k \geq 0} \frac{1}{\left\|d_k\right\|^2}=\infty$.
    To complete the proof by contradiction that there exists a positive constant $\epsilon$ such that
    \begin{equation}
    \label{ineqVk}
        \left\|v_k\right\| \geq \epsilon>0 \quad \mbox{ for all } k \geq 0 .
    \end{equation}
    By the definition of $d_k$ and the positiveness of $\beta_k$, and recall \eqref{ineq1}, we obtain that 
    \begin{equation}
        0 \leq-\psi_{x_k, v_k}(0) \leq-\psi_{x_k, d_k}(0)+\beta_k \psi_{x_{k-1}, d_{k-1}}\left(t_{k-1}d_{k-1}\right).
    \end{equation}
    Thus, considering that $t_k$ satisfies the strong Wolfe conditions, we have
    \begin{equation}
        0 \leq-\psi_{x_k, v_k}(0) \leq-\psi_{x_k, d_k}(0)-c_2 \beta_k \psi_{x_{k-1}, d_{k-1}}(0) .
    \end{equation}
    Hence, taking the square on both sides of the last inequality, we conclude that
    \begin{equation}
            \psi_{x_k, v_k}^2(0) \leq \psi_{x_k, d_k}^2(0)+c_2^2 \beta_k^2 \psi_{x_{k-1}, d_{k-1}}^2(0)+2 c_2 \psi_{x_k, d_k}(0) \beta_k \psi_{x_{k-1}, d_{k-1}}(0) .
    \end{equation}
    After some algebra in the right hand side (using $ 2ab\leq 2a^2+\frac{1}{2}b^2$ with $a=c_2 \psi_{x_k, d_k}(0) $ and $b=\beta_k \psi_{x_{k-1}, d_{k-1}}(0)$), the last inequality becomes
    \begin{equation}
    \label{ineqpsi}
        \sigma\psi_{x_k, v_k}^2(0) \leq \psi_{x_k, d_k}^2(0)+\frac{\beta_k^2}{2} \psi_{x_{k-1}, d_{k-1}}^2(0),  
    \end{equation}
    where $\sigma:=\left(1+2 c_2^2\right)^{-1}$. Then we adjust the terms to obtain
    $$
    \begin{aligned}
        \psi_{x_k, d_k}^2(0) &\geq \sigma\psi_{x_k,v_k}^2(0)-\frac{1}{2} \beta_k^2 \psi^2_{x_{k-1},d_{k-1}}(0) \\
        &\geq \sigma\psi_{x_k,v_k}^2(0)-\frac{1}{2} (\beta^{\FR}_k)^2 \psi^2_{x_{k-1},d_{k-1}}(0) \\
        & \geq \sigma \psi_{x_k v_k}^2(0)-\frac{1}{2} ({\beta^{\FR}_k})^2\left(\sigma \psi_{x_{k-1}, v_{k-1}}^2(0)-\frac{1}{2} ({\beta^{\FR}}_{k-1})^2 \varphi^2_{x_{k-2}, d_{k-2}}(0)\right) \\
        & \geq \sigma\left(\psi_{x_k, v_k}^2(0)-\frac{1}{2} ({\beta^{\FR}_k})^2 \psi_{x_{k-1}, v_{k-1}}^2+\left(-\frac{1}{2}\right)^2 (\beta^{\FR}_k)^2 {(\beta^{\FR}_{k-1})}^2 \psi^2_{x_{k-2}, v_{k-2}}(0)+\cdots \right.\\
        & \left.+\left(-\frac{1}{2}\right)^{k-1} {(\beta^{\FR}_k)}^2 \cdots {(\beta^{\FR}_2)}^2 \psi^2_{x_1, v_1}(0)\right) +\left(-\frac{1}{2}\right)^k {(\beta^{\FR}_k)}^2 \ldots {(\beta^{\FR}_1)}^2 \psi^2_{x_0, v_0}(0) \\
        & \geq \sigma\left(\psi_{x_k,v_k}^2(0)-\frac{1}{2} \psi_{x_k, v_k}^2(0)+\cdots+(-\frac{1}{2})^{k-1} \psi^2_{x_k,v_k}(0)\right)+\left(-\frac{1}{2}\right)^k \psi_{x_k, v_k}^2(0) \\
        &  \geq \sigma\psi_{x_k,v_k}^2(0)\sum_{j=1}^k\left(-\frac{1}{2}\right)^{j-1}+\left(-\frac{1}{2}\right)^k \psi_{x_k, v_k}^2(0) \\
        & \geq \frac{1}{4} \sigma\left\|v_k\right\|^4\sum_{j=1}^k\left(-\frac{1}{2}\right)^{j-1}+\left(-\frac{1}{2}\right)^k \frac{1}{4}\left\|v_k\right\|^4. \\
   \end{aligned}
    $$
    The second inequality comes from $ \beta_k \leq \beta_k^{\FR}$. The final inequality is derived as follows: because  $v_k$ is a  $K$-descent direction and $x_k$ is not a $K$-critical point, we have $\varphi(\mathrm{D}F(x_k)v_k)+\left\|v_k\right\|^2 / 2<0$, then we have 
    $ \frac{1}{4}\left\|v_k\right\|^4< \psi_{x_k,v_k}^2(0).$
    For sufficiently large values of $k$, there exists a threshold value $k_0$ such that for $k>k_0$, 
    $\left| \left(-\frac{1}{2}\right)^k \frac{1}{4} \left\| v_k \right\|^4 \right| \leq \frac{1}{8} \sigma \left\| v_k \right\|^4 \sum_{j=1}^k \left(-\frac{1}{2}\right)^{j-1}$.
    Because $1/3 < \sigma<1$ and $\sum_{j=1}^k\left(-\frac{1}{2}\right)^{j-1} \geq  1/2$, we have
     $$
    \begin{aligned}
        \psi_{x_k, d_k}^2(0) &\geq \frac{1}{4} \sigma\left\|v_k\right\|^4\sum_{j=1}^k\left(-\frac{1}{2}\right)^{j-1} - \frac{1}{8} \sigma \left\| v_k \right\|^4 \sum_{j=1}^k \left(-\frac{1}{2}\right)^{j-1} \\
        &= \frac{1}{8} \sigma \left\| v_k \right\|^4 \sum_{j=1}^k \left(-\frac{1}{2}\right)^{j-1} \\
        &\geq \frac{1}{16} \sigma\epsilon^4.
    \end{aligned}
    $$
    Thus, we conclude
    $$
    \frac{\psi_{x_k, d_k}^2(0)}{\left\|d_k\right\|^2} \geq \frac{\sigma\epsilon^4}{16\left\|d_k\right\|^2} >0,
    $$
    which is in contradiction with Zoutendjik's condition.

\end{proof}

\subsection{NVRCG method with CD parameter}

Now we consider the vector extension of the CD parameter in Riemannian optimization:
$$ \beta_k^{\CD} = \frac{\psi_{x_k,v_k}(0)}{\psi_{x_{k-1},d_{k-1}}(0)} .$$

\begin{proposition}\label{propcd}
    Assume that $\beta_k^{\CD} \geq \beta_k > 0$ and that $t_k$ satisfies the strong Wolfe conditions. If, for each $k \geq 0$, $v(x_k)\neq 0$, then the following inequality holds:
    \begin{align}\label{ineqCD}
        \psi_{x_{k},d_{k}}(0) \leq (1-c_2)\psi_{x_{k},v_{k}}(0) <  0.
    \end{align}
\end{proposition}
\begin{proof}
    Recalling the inequality \eqref{ineq1}
     and dividing by $ \psi_{x_k,v_k}(0)$, which is less than 0, on both sides, we have 
    \begin{align*}
        \frac{\psi_{x_k,d_k}(0)}{\psi_{x_k,v_k}(0)} &\geq  \frac{\psi_{x_k,v_k}(0) +\beta_k \psi_{x_{k-1},d_{k-1}}(t_{k-1}d_{k-1})}{\psi_{x_k,v_k}(0)} \\
            &= 1 + \frac{\beta_k \psi_{x_{k-1},d_{k-1}}(t_{k-1}d_{k-1})}{\psi_{x_k,v_k}(0)} \\
            &\geq 1 - c_2\beta_k\frac{\psi_{x_{k-1},d_{k-1}}(0)}{\psi_{x_k,v_k}(0)} \\
            &= 1 - c_2\frac{\beta_k}{\beta_k^{\CD}} \\
            &\geq 1-c_2.
    \end{align*}
    The second inequality comes from the strong Wolfe condition \eqref{eq:strong_wolfe} 
    and the third inequality is using the definition of $\beta_k^{\CD}$. 
    We also have the inequality that $ \psi_{x_k, v_k}(0)  < 0 $ that comes from the definition, here comes
    $ \psi_{x_{k},d_{k}}(0) \leq (1-c_2)\psi_{x_{k},v_{k}}(0) <  0$.
\end{proof}

\begin{theorem}
    Consider the NVRCG algorithm. Suppose that Assumptions \ref{assm1}, \ref{assum2} and \ref{assum3} hold and that
    \begin{equation}
    \sum_{k \geq 0} \frac{1}{\left\|d_k\right\|^2}=\infty.
    \end{equation}
    
     If $0 \leq \beta_k \leq  \beta_k^{\CD}$ and $t_k$ satisfies the strong Wolfe conditions, then $\liminf _{k \rightarrow \infty}\left\|v\left(x_k\right)\right\|=0$ holds.
\end{theorem}

\begin{proof}
    Assume by contradiction as in Theorem \ref{ThmFR} that there exists a positive constant $\epsilon$ such that
    \begin{equation*}
        \left\|v_k\right\| \geq \epsilon>0 \quad \mbox{ for all } k \geq 0 .
    \end{equation*}


    
    Because  $v_k$ is a  $K$-descent direction and $x^k$ is not a $K$-critical point, we have $\psi_{x_k,v_k}(0)+\left\|v_k\right\|^2 / 2<0$, and thus
    $ \frac{1}{4}\left\|v_k\right\|^4< \psi_{x_k,v_k}^2(0).$ Combining this with \eqref{ineqCD} we have
    $$
    \begin{aligned}
    \psi_{x_k, d_k}^2(0) &\geq (1-c_2)^2 \psi_{x_k,v_k}^2(0) \\
    & \geq \frac{1}{4}(1-c_2)^2||v_k||^4 \\
    & \geq \frac{1}{4}(1-c_2)^2\epsilon^4.
    \end{aligned}
    $$
    Then, we have
    $$
    \frac{\psi_{x_k, d_k}^2(0)}{\left\|d_k\right\|^2} \geq \frac{1}{4}(1-c_2)^2\epsilon^4\frac{1}{\left\|d_k\right\|^2} >0,
    $$
leading to a contradiction with Zoutendjik's condition.

\end{proof}

\subsection{NVRCG method with DY parameter}
Consider the $\beta^{\DY}$ vector parameter on Riemannian manifolds:
$$ \beta^{\DY}_k=\frac{-\psi_{x_k,v_k}(0)}{\psi_{x_{k-1},d_{k-1}}(t_{k-1}d_{k-1})-\psi_{x_{k-1},d_{k-1}}(0)}.$$
Using some simple derivations, we have
$$
\begin{aligned}
 \beta_k^{\DY} \psi_{x_{k-1}, d_{k-1}}(0)&=\psi_{x_{k, v_k}(0)}+\beta_k^{\DY} \psi_{x_{k-1}, d_{k-1}}\left(t_{k-1}, d_{k-1}\right) \\
& \geq \varphi\left(D F\left(x_k\right)\left(v_k+\beta_k^{\DY} \mathcal{T}^{k-1}\left(d_{k-1}\right)\right)\right) \\
& =\varphi\left(D F\left(x_k\right) d_k\right) \\
& =\psi_{x_k, d_k} (0).
\end{aligned}
$$
If $\psi_{x_{k-1}, d_{k-1}}(0)<0$,
then it holds that 
\begin{equation}\label{ineqDY}
    \beta_k^{\DY} \leq \frac{\psi_{x_k, d_k}(0)}{\psi_{x_{k-1}, d_{k-1}}(0)}.
\end{equation}

\begin{proposition}
    Assume $0\leq \beta_k \leq \beta_k^{\DY}$ and that $t_k$ satisfies the Wolfe conditions. If, for each $k \geq 0$, $v(x_k)\neq 0$, then the algorithm is well defined and the following two 
    inequalities hold:
    \begin{align}
        \psi_{x_k, d_k}(0) &\leq 0, \\
        \psi_{x_{k},d_{k}}(0) &\leq \psi_{x_{k},d_{k}}(t_{k}d_{k}).
    \end{align}
\end{proposition}
\begin{proof}
It can be proved by induction as in \cite{sato_riemannian_2021}.

 When $k = 0$, we have $\psi_{x_0, d_0}(0) = \psi_{x_0, v_0}(0)\leq 0$ and 
 $\psi_{x_{0},d_{0}}(t_{0}d_{0}) \geq c_2 \psi_{x_{0},d_{0}}(0) \geq \psi_{x_{0},d_{0}}(0). $

Assume these two inequalities still hold for arbitrary $k -1$. Then $\beta_k^{\DY} \leq 0$ holds, which means it is well defined. Recalling inequality \eqref{ineq1}, combined with the definition of $\beta^{\DY}$, we can easily deduce
    \begin{align*}
        \psi_{x_k,d_k}(0) &\leq \psi_{x_k,v_k}(0) +\beta_k \psi_{x_{k-1},d_{k-1}}(t_{k-1}d_{k-1}) \\
        &= \psi_{x_k,v_k}(0) +\beta_k [\psi_{x_{k-1},d_{k-1}}(t_{k-1}d_{k-1})- \psi_{x_{k-1},d_{k-1}}(0)] + \beta_k\psi_{x_{k-1},d_{k-1}}(0)\\
        &\leq \psi_{x_k,v_k}(0) +\beta_k^{\DY} [\psi_{x_{k-1},d_{k-1}}(t_{k-1}d_{k-1})- \psi_{x_{k-1},d_{k-1}}(0)] + \beta_k\psi_{x_{k-1},d_{k-1}}(0)\\
        &= \beta_k\psi_{x_{k-1},d_{k-1}}(0) \\
        &<0.
    \end{align*}
About the second inequality, considering Wolfe condition \eqref{eq:wolfe}, we have 
$$
\psi_{x_{k},d_{k}}(t_{k}d_{k}) \geq c_2\psi_{x_{k},d_{k}}(0) \geq \psi_{x_{k},d_{k}}(0).
$$

\end{proof}

\begin{theorem}
\label{ThmDY}
    Consider the NVRCG algorithm. Assume that Assumptions \ref{assm1}, \ref{assum2} and \ref{assum3}  hold and that
    \begin{equation}
    \sum_{k \geq 0} \frac{1}{\left\|d_k\right\|^2}=\infty.
    \end{equation}
    
     If $ 0 \leq \beta_k \leq  \beta_k^{\DY}$ and $t_k$ satisfies the strong Wolfe conditions, then $\liminf _{k \rightarrow \infty}\left\|v\left(x_k\right)\right\|=0$ holds.
\end{theorem}
\begin{proof}
    Assume by contradiction as discussed in Theorem \ref{ThmFR} that the inequalities \eqref{ineqVk} and \eqref{ineqpsi} still hold.
    Then we adjust the terms of \eqref{ineqpsi} and use \eqref{ineqDY} to obtain
    $$
    \begin{aligned}
        \psi_{x_k, d_k}^2(0)&\geq \sigma\psi_{x_k, v_k}^2(0) - \frac{\beta_k^2}{2} \psi_{x_{k-1}, d_{k-1}}^2(0) \\
    & \geq  \sigma\psi_{x_k, v_k}^2(0) - \frac{1}{2}\frac{\psi^2_{x_k, d_k}(0)}{\psi^2_{x_{k-1}, d_{k-1}}(0)}\psi_{x_{k-1}, d_{k-1}}^2(0) \\
    & = \sigma\psi_{x_k, v_k}^2(0) - \frac{1}{2}\psi^2_{x_k, d_k}(0).
    \end{aligned}
    $$

    Because  $v_k$ is a  $K$-descent direction and $x^k$ is not a ${K}$-critical point, we have $\varphi(\mathrm{D}F(x_k)v_k)+\left\|v_k\right\|^2 / 2<0$, and hence
    $ \frac{1}{4}\left\|v_k\right\|^4< \psi_{x_k,v_k}^2(0).$ Therefore, it holds that 
    $$
    \psi_{x_k, d_k}^2(0) \geq \frac{2}{3}\sigma\psi_{x_k, v_k}^2(0) \geq \frac{1}{6}\sigma\left\|v_k\right\|^4.
    $$
    Then we have
    $$
    \frac{\psi_{x_k, d_k}^2(0)}{\left\|d_k\right\|^2} \geq \frac{\sigma\epsilon^4}{6\left\|d_k\right\|^2} >0,
    $$
    which leads to a contradiction with Zoutendjik's condition.
\end{proof}

\subsection{NVRCG method with PRP, HS, and LS parameter}
While the FR, DY, and CD parameters of conjugate gradient methods have been extended, the theoretical properties of the PRP, HS, and LS parameters remain less explored. Even in Euclidean spaces, these three conjugate gradient methods with the (strong) Wolfe conditions do not guarantee convergence \cite{andrei2020nonlinear}. Powell~\cite{powell1984nonconvex} demonstrated that the PRP and HS parameters combined with exact line searches can exhibit cycling behavior without approaching a solution point. As a result, various variants have been proposed and analyzed.
Gilbert and Nocedal \cite{gilbert_global_1992} proved that global convergence under inexact line search can be achieved for \(\beta_k = \max\{\beta^{\PRP}_k, 0\}\) and \(\beta_k = \max\{\beta^{\HS}_k, 0\}\).  Dai et al. \cite{dai2000convergence} showed that the positiveness restriction on \(\beta_k\) cannot be relaxed for the PRP parameter.  
There is a large variety of hybrid conjugate gradient methods whose purpose is to combine the properties of the standard ones in order to get new ones, rapidly convergent to a solution. The idea is to avoid jamming. For example, the method with FR parameter has strong convergence properties, but it may not perform well in computational experiments. On the other hand, although the method with PRP parameter may not generally converge, they often perform better than FR. Indeed, the method with PRP parameter possesses a built-in restart feature that directly addresses to jamming. When the difference between $x_k$ and $x_{k+1}$ is small enough, the numerator of $ \beta^{\PRP}$ tends to zero. In this case, the search direction $d_k$  is essentially the steepest descent direction just after a restart occurred.
Therefore, the combination of these methods tries to exploit the attractive features of each one, which is the motivation for the hybrid computational scheme PRP--FR.
The PRP--, HS--, and LS--type hybrid schemes of conjugate gradient methods to generate descent directions have been discussed in \cite{andrei2020nonlinear, sato_riemannian_2021_article}. 
Here, we follow the hybrid idea and simply propose the following algorithms. 

Recall the PRP  parameter
$$
\beta^{\PRP}_k = \frac{-\psi_{x_k,v_k}(0)+\psi_{x_{k},\mathcal{T}^{k-1}{(v_k)}}(0)}{-\psi_{x_{k-1},v_{k-1}}(0)}.
$$

\begin{theorem}
    Under the assumptions for $\beta^{\FR}$, we set $\beta_k=\max\{0, \min\{\beta_k^{\FR},\beta_k^{\PRP}\}\}$. Then $\liminf _{k \rightarrow \infty}\left\|v\left(x_k\right)\right\|=0$ holds.
\end{theorem}

Recall the LS parameter
$$ \beta_k^{\LS} :=\frac{-\psi_{x_k,v_k}(0)+\psi_{x_{k-1},\mathcal{T}^{k-1}{(v_k)}}(0)}{-\psi_{x_{k-1},d_{k-1}}(0)} .$$


    


\begin{theorem}
    Under the assumptions for $\beta^{\CD}$, we set $\beta_k=\max\{0, \min\{\beta_k^{\LS},\beta_k^{\CD}\}\}$. Then $\liminf _{k \rightarrow \infty}\left\|v\left(x_k\right)\right\|=0$ holds.
\end{theorem}
Recall the HS  parameter
$$
\beta^{\HS}_k = \frac{-\psi_{x_k,v_k}(0)+\psi_{x_{k-1},\mathcal{T}^{k-1}{(v_k)}}(0)}{\psi_{x_{k},\mathcal{T}^{k-1}{(d_{k-1})}}(0)-\psi_{x_{k-1},d_{k-1}}(0)}.
$$

\begin{theorem}
    Under the assumptions for $\beta^{\DY}$, we set $\beta_k=\max\{0, \min\{\beta_k^{\HS},\beta_k^{\DY}\}\}$. Then $\liminf _{k \rightarrow \infty}\left\|v\left(x_k\right)\right\|=0$ holds.
\end{theorem}




\section{Numerical experiments}
\label{sec:experiments}

In this section, we present the numerical results to illustrate the NVRCG algorithm for the different choices of parameter $\beta$ on the sphere, which is a typical manifold. 

We solve some simple problems with the steepest descent and conjugate gradient methods with FR, CD, and DY parameters. For each method, we applied the standard  Wolfe conditions. We choose a random initial point $x_0$ on each manifold and terminate the algorithms when  $ \|v(x)\|\leq 10^{-4}$ or $ t\leq 10^{-4}$ is attained and run the algorithms 100 times for each method. We implemented the code in Python 3.9.12 on an Apple MacBook Pro  with Apple M1 Pro chip and 16GB of RAM.

Consider the problem:
\begin{equation}\label{exprob:1}
    \begin{array}{ll}
    \operatorname{min} & F(x) = (f_1(x), f_2(x))^T, \\
    \text { s.t. } & x \in \mathcal{S}^1,
    \end{array}
\end{equation}
where $f_1(x)=x^{\top} A x$, $f_2(x)=c x$ and $\mathcal{S}^1$ is a $1$--dimensional sphere (circle).
Let
$
A_1=\left[\begin{array}{ll}
1 & 0 \\
0 & 1
\end{array}\right]
$,
$
A_2=\left[\begin{array}{ll}
1 & 1 \\
1 & 1
\end{array}\right]
$, 
$
A_3=\left[\begin{array}{ll}
1 & 2 \\
2 & 2
\end{array}\right]
$, and $c=[1,1]$.

Here we use three cases of $A$ to show three different situations of convergence points. In case $A_1$, convergence to a single optimal point is observed. In the case of $A_2$, the Pareto stationary points are also Pareto optimal points. Case $A_3$ shows a more general result.

Although the results presented in Table \ref{tab:1-1} with $A_2$ and $A_3$ do not exhibit significant differences due to the limited number of iterations, we can still see that the conjugate gradient methods have fewer cycles than the steepest descent method. From the results provided in Table \ref{tab:1-1} with $A_1$, it is evident that the conjugate method outperforms the steepest descent method by significantly reducing the number of iterations. Particularly noteworthy is the remarkable performance of the algorithm when employing the DY parameter. 

Furthermore, it is noteworthy that Figure \ref{Fig3.sub.2} exhibits the presence of two distinct Pareto fronts, but clearly, the result of the lower one is superior to the upper one. This phenomenon arises due to the convergence of points that correspond to Pareto stable solutions rather than Pareto optimal solutions. In many multiobjective optimization algorithms in Euclidean space, the prevailing assumption is the convexity of the objective functions, which ensures that the Pareto stable solutions coincide with the Pareto optimal solutions. However, the treatment of scenarios where convexity assumptions are relaxed, and the handling of stable solutions, remains an unresolved topic, particularly in the context of Riemannian optimization.  

\begin{table}[ht]
    \centering
\caption{Problem \eqref{exprob:1} with $c_1 = 0.1 $ and $c_2 = 0.6$}
\label{tab:1-1}
\begin{tabular}{|c|c|c|llll|}
\hline
\multirow{2}{*}{A}     & \multirow{2}{*}{n} & \multirow{2}{*}{Times} & \multicolumn{4}{c|}{Average iterations}  \\ \cline{4-7} 
                       &      &      & FR & CD & DY & SD \\ \hline
$A_1$                    &  2   & 100  &  12.17  &  17.73  &  6.95  & 29.47   \\ \hline
$A_2 $                   &  2   & 100  &  1.85  &  2.21  &  1.79  &  3.87  \\ \hline
$A_3$                    &  2   & 100  &  2.50  &  1.60  &  1.91  &  1.83  \\ \hline
\end{tabular}
\end{table}

    \begin{figure}[ht]
    \centering  
    \subfigure[Function images]{
    \label{Fig1.sub.1}
    \includegraphics[width=0.45\textwidth]{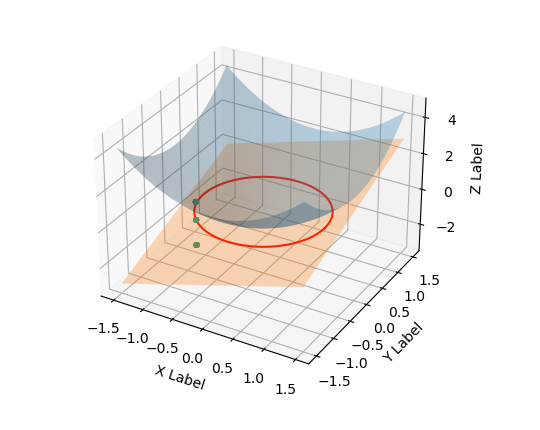}}
    \subfigure[Pareto stationary points]{
    \label{Fig1.sub.2}
    \includegraphics[width=0.5\textwidth]{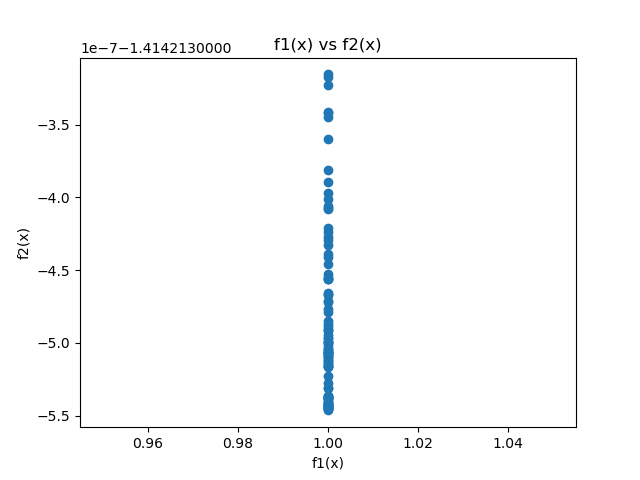}}
    \label{fig: 1}
    \caption{Problem \eqref{exprob:1} with $
    A_1
    $
    and $c$.}
    \end{figure}

    \begin{figure}[ht]
        \centering  
        \subfigure[Function images]{
        \label{Fig2.sub.1}
        \includegraphics[width=0.45\textwidth]{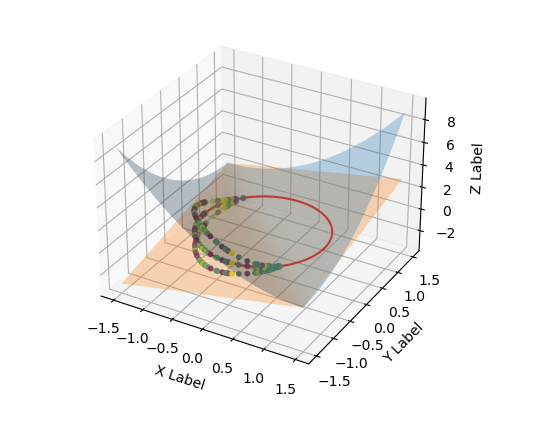}}
        \subfigure[Pareto stationary points]{
        \label{Fig2.sub.2}
        \includegraphics[width=0.5\textwidth]{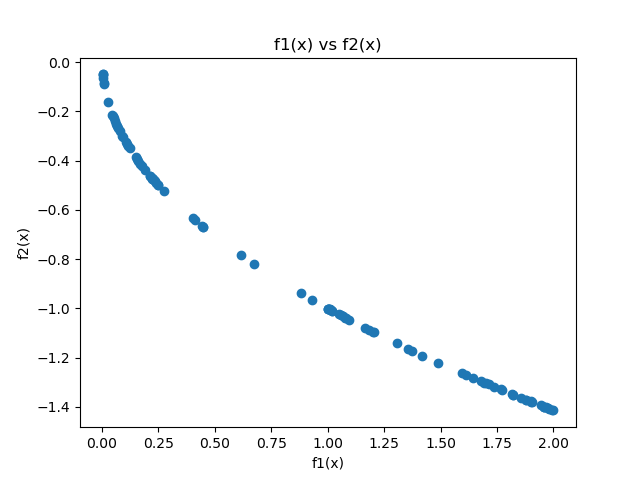}}
        \caption{Problem \eqref{exprob:1} with $
        A_2$
        and $c$.}
        \label{Fig.main}	
    \end{figure}

    \begin{figure}[ht]
	\centering  
	\subfigure[Function images]{
        \label{Fig3.sub.1}
	\includegraphics[width=0.45\textwidth]{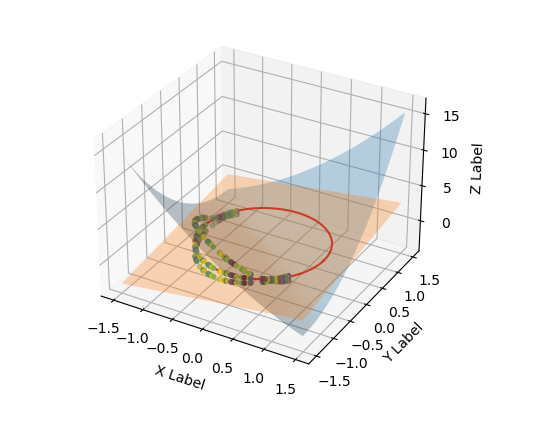}}
	\subfigure[Pareto stationary points]{
        \label{Fig3.sub.2}
	\includegraphics[width=0.5\textwidth]{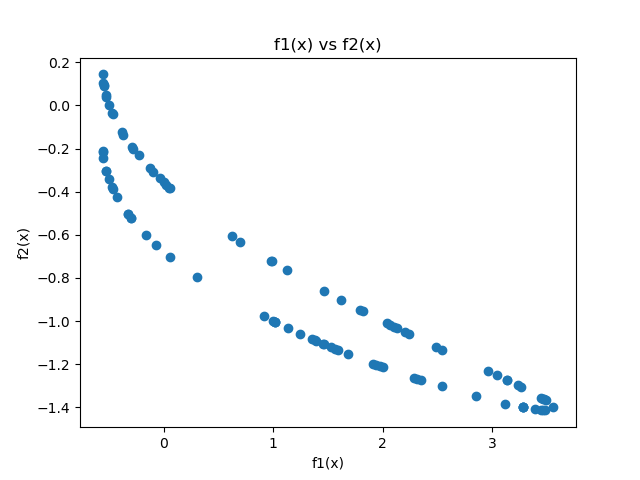}}
	\caption{Problem \eqref{exprob:1} with $A_3$ and $c$.}
        \label{fig3}
  \end{figure}

Now consider another problem:
\begin{equation}\label{exprob:2}
    \begin{array}{ll}
    \operatorname{min} & F(x) = (f_1(x), f_2(x))^T, \\
    \text { s.t. } & x \in \mathcal{S}^{n-1},
    \end{array}
\end{equation}
where $f_i(x)=x^{\top} A_i x$ for all $i $ and $\mathcal{S}^{n-1}$ is a sphere manifold.

The results can be seen in Table \ref{tab:2-1}, where we do the experiments with $c_1=0.001$, $c_2=0.6$. In order to better compare the performance of the algorithms, we set the same initial point every time.
It can be seen from the table that the advantage of the conjugate gradient descent method is not significant when dealing with low-dimensional problems.
For FR, CD, and DY parameters, in the case of $n= 2$, the average number of iterations for the conjugate gradient method is relatively low, between 3 and 4.  This may be because the problem has a lower dimension and the algorithm can find the optimal solution faster, but still faster than the steepest descent method. This disparity becomes increasingly pronounced as the dimension of the problem increases. Based on the data presented in Table \ref{tab:2-1}, it is evident that the conjugate gradient method exhibits a notably reduced number of iterations compared to the steepest descent method, particularly when employing the DY parameter. As for the hybrid algorithms, as the dimension increases, its superiority over the steepest descent method becomes increasingly pronounced. However, in comparison to the algorithms with FR, CD, or DY parameter, it still exhibits a higher number of iterations and does not possess a distinct advantage. These findings underscore the effectiveness and efficiency of the conjugate method in comparison to the steepest descent method, particularly under the influence of the DY parameter.

\begin{table}[ht]
\centering
\caption{Problem \eqref{exprob:2} with $c_1=0.001$, $c_2=0.6$}
\label{tab:2-1}
\begin{tabular}{|c|c|ccccccc|}
\hline
\multirow{2}{*}{n} & \multirow{2}{*}{Times} & \multicolumn{7}{c|}{Average iterations} \\ \cline{3-9} 
 &  & {FR} & {CD} & {DY} & {PRP-} & {LS-} & {HS-} & SD \\ \hline
  2&  100& 3.46 & 3.44 & 3.68 & 4.75 & 4.40 & 3.92 & 4.41 \\ \hline
  5&  100& 7.38 & 7.44 & 4.03 & 4.98 & 5.08 & 4.50  &  5.17\\ \hline
 10&  100& 10.22 & 7.47 & 4.59 & 7.19 & 7.19 & 6.93  & 11.07 \\ \hline
 100& 100& 9.73 & 8.83 & 4.09 & 15.38 & 15.29 & 15.82 & 33.18 \\ \hline
 200& 100& 8.11 & 8.63 & 4.13 & 18.10 & 18.24 & 16.44 & 31.05 \\ \hline
\end{tabular}
\end{table}

\section{Conclusion and future work}
\label{sec:conclusion}

In this paper, we proposed the nonlinear conjugate gradient methods for vector optimization on Riemannian manifolds with retraction and vector transport. We gave the extension of FR, CD, DY, PRP, LS, and HS parameters. We proved the convergence with the vector extensions of the FR, CD, and DY parameters, then simply give results of PRP, HS, and LS parameters. Under some assumptions, the sequence obtained by the algorithm can converge to a Pareto stationary point. Numerical experiments illustrating the practical behavior of the methods were presented.

For future work, the convergence with the exact PRP, HS, and LS parameters will be discussed. Additionally, an analysis of the convergence rate will be conducted. 
Furthermore, the treatment of Pareto stationary points in order to mitigate the influence of local optimum remains an intriguing area of research, particularly when relaxing the conventional convexity assumption. Further investigation is warranted to develop effective strategies for handling Pareto stationary points and effectively excluding them to achieve improved convergence toward global optimization.


\section*{Acknowledgments} 
\noindent This work was supported by JST, the establishment of university 
fellowships towards the creation of science technology innovation 
(JPMJFS2123), and Japan Society for the Promotion of Science, 
Grant-in-Aid for Scientific Research (C) (19K11840) and Grant-in-Aid for 
Early-Career Scientists (20K14359).

\printbibliography[heading=bibintoc, title=\ebibname]

\appendix
\addappheadtotoc

\end{document}